\documentclass{amsart}
\usepackage{amssymb}
\usepackage{eurosym}
\usepackage{amsmath}
\usepackage{amsfonts}

\newtheorem{theorem}{Theorem}
\theoremstyle{plain}

\newtheorem{corollary}{Corollary}

\newtheorem{definition}{Definition}

\newtheorem{lemma}{Lemma}

\newtheorem{remark}{Remark}

\numberwithin{equation}{section}
\input{tcilatex}

\begin{document}
\title{Non-Null Weakened Mannheim \ Curves in Minkowski 3-Space}
\author{Y\i lmaz Tun\c{c}er}
\address{U\c{s}ak University, Faculty of Sciences and Arts, Department of
Mathematics. 1 Eylul Campus-64200,U\c{s}ak-TURKEY}
\email{yilmaz.tuncer@usak.edu.tr}
\author{Murat Kemal Karacan}
\address{U\c{s}ak University, Faculty of Sciences and Arts, Department of
Mathematics. 1 Eylul Campus-64200,U\c{s}ak-TURKEY}
\email{murat.karacan@usak.edu.tr}
\author{Dae Won Yoon}
\address{Department of Mathematics Education and RINS,Gyeongsang National
University, Jinju 660-701, South Korea}
\email{dwyoon@gnu.ac.kr}
\thanks{2000 Mathematics Subject Classifications:53A35,53B30}
\keywords{Mannheim Curves,Frenet-Mannheim Curves,Weakened-Mannheim Curves, }

\begin{abstract}
{\small \ }In this study, Non-null Frenet-Mannheim curves and Non-null
Weakened Mannheim curves are investigated \ in Minkowski 3-space. Some
characterizations for this curves are obtained.
\end{abstract}

\maketitle

\section{\protect\bigskip Introduction}

In the study of the fundamental theory and the characterizations of space
curves, the corresponding relations between the curves are the very
interesting and important problem. The well-known Bertrand curve is
characterized as a kind of such corresponding relation between the two
curves. For the Bertrand curve $\alpha $, it shares the normal lines with
another curve $\beta $, called Bertrand mate or Bertrand partner curve of $%
\alpha $ [2].

In 1967, H.F.Lai investigated the properties of two types of similar curves
(the Frenet-Bertrand curves and the Weakened Bertrand curves) under weakened
conditions.

In recent works, Liu and Wang (2007, 2008) are curious about the Mannheim
curves in both Euclidean and Minkowski 3-space and they obtained the
necessary and sufficient conditions between the curvature and the torsion
for a curve to be the Mannheim partner curves. Meanwhile, the detailed
discussion concerned with the Mannheim curves can be found in literature
(Wang and Liu, 2007; Liu and Wang, 2008; Orbay and Kasap, 2009) and
references there in [3].

In this paper, our main purpose is to carry out some results which were
given in [1] to Frenet-Mannheim curves and Weakened Mannheim curves.

\section{Preliminaries}

Let we denote the non-null curves $\alpha (s),$ $s\in I$ and $\beta
(s^{\star }),$ $s^{\star }\in I^{\star }$ in $E_{v}^{3}$ as $\Gamma $ and $%
\widetilde{\Gamma }$\ respectively, where $v=1,2,3$.

An arbitrary curve $\Gamma $ is called a space-like, light-like or time-like
if $\ $all of its velocity vectors $\alpha (s)^{\prime }$ space-like or
time-like respectively. Let $T$ , $N$ and $B$ be tangent, principal normal
and binormal vector fiels of $\Gamma .$ If $\Gamma $ is a spacelike curve
with a spacelike or timelike principal normal $N$, then the Frenet formulae
read

\begin{equation}
\left[ 
\begin{array}{c}
T^{\prime } \\ 
N^{\prime } \\ 
B^{\prime }%
\end{array}%
\right] =\left[ 
\begin{array}{ccc}
0 & \kappa & 0 \\ 
-\varepsilon \kappa & 0 & \tau \\ 
0 & \tau & 0%
\end{array}%
\right] \left[ 
\begin{array}{c}
T \\ 
N \\ 
B%
\end{array}%
\right]  \tag{2.1}
\end{equation}%
where $\left\langle T,T\right\rangle =1,$ $\left\langle N,N\right\rangle
=-\left\langle B,B\right\rangle =\varepsilon =\pm 1$. If $\Gamma $ is a
timelike curve then the Frenet formulae read

\begin{equation}
\left[ 
\begin{array}{c}
T^{\prime } \\ 
N^{\prime } \\ 
B^{\prime }%
\end{array}%
\right] =\left[ 
\begin{array}{ccc}
0 & \kappa & 0 \\ 
\kappa & 0 & \tau \\ 
0 & -\tau & 0%
\end{array}%
\right] \left[ 
\begin{array}{c}
T \\ 
N \\ 
B%
\end{array}%
\right]  \tag{2.2}
\end{equation}%
where $\left\langle T,T\right\rangle =-1,$ $\left\langle N,N\right\rangle
=\left\langle B,B\right\rangle =1$ for some $C^{\infty }$ scalar functions $%
\kappa (s),\tau (s)$ which are called the curvature and torsion of $\Gamma $%
. We can state both (2.1) and (2.2) as%
\begin{equation}
\left[ 
\begin{array}{c}
T^{\prime } \\ 
N^{\prime } \\ 
B^{\prime }%
\end{array}%
\right] =\left[ 
\begin{array}{ccc}
0 & \kappa & 0 \\ 
-\varepsilon _{1}\varepsilon _{2}\kappa & 0 & \tau \\ 
0 & \varepsilon _{1}\tau & 0%
\end{array}%
\right] \left[ 
\begin{array}{c}
T \\ 
N \\ 
B%
\end{array}%
\right]  \tag{2.3}
\end{equation}%
for the non-null space curve, where $\left\langle T,T\right\rangle
=\varepsilon _{1}=\pm 1$ and $\left\langle N,N\right\rangle =-\left\langle
B,B\right\rangle =\varepsilon _{2}=\pm 1.$[5].

In this paper we weaken the conditions. We shall adopt the definition of a $%
C^{\infty }$ Frenet curve as a $C^{\infty }$ regular curve $\Gamma $, for
which there exists a $C^{\infty }$ family of Frenet frames, that is,
right-handed orthonormal frames $\left\{ T,N,B\right\} $ where $T=\alpha
^{\prime }$, satisfying the Frenet equations.

\begin{definition}
Let $E_{v}^{3}$ be the 3-dimensional Minkowski space with the standard inner
product $\left\langle ,\right\rangle $. If there exists a corresponding
relationship between the space curves $\Gamma $ and $\widetilde{\Gamma }$
such that, at the corresponding points of the curves, the principal normal
lines of $\Gamma $ coincides with the binormal lines of $\widetilde{\Gamma }$%
, then $\Gamma $ is called a Mannheim curve, and $\widetilde{\Gamma }$ a
Mannheim partner curve of $\Gamma $. The pair $(\Gamma ,\widetilde{\Gamma })$
is said to be a Mannheim pair.\bigskip
\end{definition}

\begin{definition}
A Mannheim curve $\Gamma $ is a $C^{\infty }$ regular curve with non-zero
curvature for which there exists another (different) $C^{\infty }$ regular
curve $\widetilde{\Gamma }$ where $\widetilde{\Gamma }$ is of class $%
C^{\infty }$ and $\beta ^{\prime }(s^{\star })\neq 0$ ($s$ being the arc
length of $\Gamma $ only), also with non-zero curvature, in bijection with
it in such a manner that the principal normal to $\Gamma $ and the binormal
to $\widetilde{\Gamma }$ at each pair of corresponding points coincide with
the line joining the corresponding points. The curve $\widetilde{\Gamma }$
is called a Mannheim conjugate of $\Gamma $.
\end{definition}

We can state five cases for the non-null Mannheim parners $(\Gamma ,%
\widetilde{\Gamma })$ with respect o theirs causal caracterizations by using
the linearly dependence principal vectors of $\Gamma $ and binormal vectors
of $\widetilde{\Gamma }.$ The cases are

\textbf{Case 1:} If $\Gamma $\ is a timelike curve and it's Mannheim partner
is a timelike curve,

\textbf{Case 2:} If $\Gamma $\ is a timelike curve and it's Mannheim partner
is a spacelike curve with the timelike principal normal vector,

\textbf{Case 3:} If $\Gamma $\ is a spacelike curve with the timelike
principal normal vector and it's Mannheim partner is a spacelike curve with
the spacelike principal normal vector,

\textbf{Case 4:} If $\Gamma $\ is a spacelike curve with the spacelike
principal normal vector and it's Mannheim partner is a spacelike curve with
the timelike principal normal vector,

\textbf{Case 5:} If $\Gamma $\ is a spacelike curve with the spacelike
principal normal vector and it's Mannheim partner is a timelike curve.

\begin{definition}
A non-null Frenet-Mannheim curve $\Gamma $ (briefly called a NN-FM curve) is
a $C^{\infty }$ Frenet curve for which there exists another $C^{\infty }$
Frenet curve $\widetilde{\Gamma }$, where $\widetilde{\Gamma }$ is of class $%
C^{\infty }$ and $\beta ^{\prime }(s^{\star })\neq 0$, in bijection with it
so that, by suitable choice of the Frenet frames the principal normal vector 
$N_{\Gamma }(s)$ and Binormal vector $B_{\widetilde{\Gamma }}(s^{\ast })$ at
corresponding points on $\Gamma ,\widetilde{\Gamma }$ both lie on the line
joining the corresponding points. The curve $\widetilde{\Gamma }$ is called
a NN-FM conjugate of $\Gamma $.
\end{definition}

\begin{definition}
\bigskip A non-null weakened Mannheim curve $\Gamma $ (briefly called a
NN-WM curve) is a $C^{\infty }$ regular curve for which there exists another 
$C^{\infty }$ regular curve $\widetilde{\Gamma }:\beta (s^{\star })$,$s\in
I^{\star }$ , where $s^{\star }$ is the arclength of $\widetilde{\Gamma }$,
and a homeomorphism $\sigma :I\rightarrow I^{\star }$ such that

$(i)$ Here exist two (disjoint) closed subsets $Z,N$ of $I$ with void
interiors such that $\sigma \in C^{\infty }$ on $L\backslash N,$ $\left( 
\frac{ds^{\star }}{ds}\right) =0$ on $Z$, $\sigma ^{-1}\in C^{\infty }$ on $%
\sigma \left( L\backslash Z\right) $ and $\left( \frac{ds}{ds^{\star }}%
\right) =0$ on $\sigma (N).$

$(ii)$ The line joining corresponding points $s,s^{\star }$ of \ $\Gamma $
and $\widetilde{\Gamma }$ is orthogonal to $\Gamma $ and $\widetilde{\Gamma }
$ at the points $s,s^{\star }$ respectively, and is along the principal
normal to $\Gamma $ or $\widetilde{\Gamma }$ at the points $s,s^{\star }$
whenever it is well defined. The curve $\widetilde{\Gamma }$ is called a WM
conjugate of $\Gamma $.
\end{definition}

Thus for a NN-WM curve we not only drop the requirement of $\Gamma $ being a
Frenet curve, but also allow $\left( \frac{ds^{\star }}{ds}\right) $ to be
zero on a subset with void interior $\left( \frac{ds^{\star }}{ds}\right) =0$
on an interval would destroy the injectivity of the mapping $\sigma $. Since 
$\left( \frac{ds^{\star }}{ds}\right) =0$ implies that $\left( \frac{ds}{%
ds^{\star }}\right) $ does not exist, the apparently artificial requirements
in $(i)$ are in fact quite natural.

\section{Non-Null Frenet-Mannheim curves}

In this section we study the structure and characterization of NN-FM curves.
We begin with a lemma, the method used in which is classical.

\begin{lemma}
\bigskip Let $\Gamma $ be a NN-FM curve and $\widetilde{\Gamma }$ a NN-FM
conjugate of $\Gamma $. Let all quantities belonging to $\widetilde{\Gamma }$
be marked with a tilde. Let 
\begin{equation}
\alpha (s)=\beta (s^{\star })+\lambda (s^{\star })B_{\widetilde{\Gamma }%
}(s^{\star })  \tag{3.1}
\end{equation}%
Then the distance $\left\vert \lambda \right\vert $ between corresponding
points is constant for all cases, and the torsion of $\widetilde{\Gamma }$
is 
\begin{equation*}
\tau _{\widetilde{\Gamma }}=\frac{-\widetilde{\varepsilon }_{1}\tau _{\Gamma
}}{1+\varepsilon _{1}\varepsilon _{2}\lambda \mu \kappa _{\Gamma }}
\end{equation*}%
and 
\begin{equation*}
\begin{array}{cc}
(i)\text{ \ }\sinh \varphi =\lambda \widetilde{\varepsilon }_{1}\tau _{%
\widetilde{\Gamma }}\cosh \varphi \text{ \ \ \ \ \ \ \ \ \ \ \ \ \ \ \ \ \ \
\ \ \ \ \ \ \ \ } & (iii)\text{ \ }\cosh ^{2}\varphi =1+\varepsilon
_{1}\varepsilon _{2}\lambda \mu \kappa _{\Gamma } \\ 
(ii)\text{ \ }\left( 1+\varepsilon _{1}\varepsilon _{2}\lambda \mu \kappa
_{\Gamma }\right) \sinh \varphi =-\lambda \mu \tau _{\Gamma }\cosh \varphi & 
(iv)\text{ \ }\sinh ^{2}\varphi =-\lambda ^{2}\mu \widetilde{\varepsilon }%
_{1}\tau _{\Gamma }\tau _{\widetilde{\Gamma }}\text{ \ \ }%
\end{array}%
\end{equation*}%
are satisfies for the cases 1,2,4 and 5. For the case 3, the torsion of $%
\widetilde{\Gamma }$ is%
\begin{equation*}
\tau _{\widetilde{\Gamma }}=\frac{-\tau _{\Gamma }}{1-\lambda \mu \kappa
_{\Gamma }}.
\end{equation*}%
and%
\begin{equation*}
\begin{array}{cc}
(v)\text{ \ }\sin \varphi =\lambda \tau _{\widetilde{\Gamma }}\cos \varphi 
\text{ \ \ \ \ \ \ \ \ \ \ \ \ \ \ \ \ \ \ \ \ } & (vii)\text{ \ }\cos
^{2}\varphi =1-\lambda \mu \kappa _{\Gamma } \\ 
(vi)\text{ \ }\left( 1-\lambda \mu \kappa _{\Gamma }\right) \sin \varphi
=\lambda \mu \tau _{\Gamma }\cos \varphi & (viii)\text{ \ }\sin ^{2}\varphi
=\lambda ^{2}\mu \tau _{\Gamma }\tau _{\widetilde{\Gamma }}\text{ \ \ }%
\end{array}%
\end{equation*}%
are satisfies.
\end{lemma}

\begin{proof}
\bigskip\ From (3.1) it follows that 
\begin{equation*}
\lambda (s^{\star })=\left\langle \alpha (s)-\beta (s^{\star }),B_{%
\widetilde{\Gamma }}(s^{\star })\right\rangle
\end{equation*}%
is of class $C^{\infty }$. Differentiation of (3.1) with respect to $%
s^{\star }$ gives%
\begin{equation}
T_{\Gamma }\frac{ds}{ds^{\star }}=T_{\widetilde{\Gamma }}+\lambda \widetilde{%
\varepsilon }_{1}\tau _{\widetilde{\Gamma }}N_{\widetilde{\Gamma }}-\lambda
^{\prime }B_{\widetilde{\Gamma }}  \tag{3.2}
\end{equation}%
where $\left\langle T_{\widetilde{\Gamma }},T_{\widetilde{\Gamma }%
}\right\rangle =\widetilde{\varepsilon }_{1}=\pm 1.$ Since by hypothesis we
have $B_{\widetilde{\Gamma }}=\mu N_{\Gamma }$ with $\mu =\pm 1$, scalar
multiplication of (3. 2) by $B_{\Gamma }$ gives%
\begin{equation*}
\lambda ^{\prime }=0\Rightarrow \lambda =\text{cons}\tan \text{t.}
\end{equation*}%
thus $d\left( \alpha (s),\beta (s^{\star })\right) $ is constant. Since $%
\left\langle T_{\Gamma },B_{\widetilde{\Gamma }}\right\rangle =\left\langle
N_{\Gamma },N_{\widetilde{\Gamma }}\right\rangle =\left\langle B_{\Gamma
},B_{\widetilde{\Gamma }}\right\rangle =0$ then we can write,%
\begin{eqnarray}
T_{\Gamma } &=&\cosh \varphi T_{\widetilde{\Gamma }}+\sinh \varphi N_{%
\widetilde{\Gamma }}  \TCItag{3.3} \\
B_{\Gamma } &=&\sinh \varphi T_{\widetilde{\Gamma }}+\cosh \varphi N_{%
\widetilde{\Gamma }}  \notag
\end{eqnarray}%
by using linearly dependance of $B_{\widetilde{\Gamma }}$ and $N_{\Gamma }$
where $\varphi =\varphi (s^{\star })$ is hyperbolic angle between the
tangents of $\Gamma $ and $\widetilde{\Gamma }$. By definition of NN-FM
curve we have $\frac{ds^{\star }}{ds}\neq 0$ (also $\frac{ds^{\star }}{ds}%
=\cosh \varphi \neq 0$), so that $T_{\widetilde{\Gamma }}$ (also $T_{\Gamma
} $) is $C^{\infty }$ function of $s.~$From (3.2) and (3.3) we obtain 
\begin{equation}
\frac{ds^{\star }}{ds}=\cosh \varphi \text{ \ \ and \ \ }\frac{ds^{\star }}{%
ds}\lambda \widetilde{\varepsilon }_{1}\tau _{\widetilde{\Gamma }}=\sinh
\varphi .  \tag{3.4}
\end{equation}%
Since $B_{\widetilde{\Gamma }}=\mu N_{\Gamma }$ then we can write%
\begin{equation*}
\alpha (s)=\beta (s^{\star })+\mu \lambda N_{\Gamma }(s)
\end{equation*}%
and also 
\begin{equation}
\beta (s^{\star })=\alpha (s)-\mu \lambda N_{\Gamma }(s).  \tag{3.5}
\end{equation}

By taking derivative (3.5) with respect to $s$ and compare with the invers
of (3.3) we get,%
\begin{equation}
T_{\widetilde{\Gamma }}=\frac{ds}{ds^{\star }}\left[ \left( 1+\varepsilon
_{1}\varepsilon _{2}\lambda \mu \kappa _{\Gamma }\right) T_{\Gamma }-\lambda
\mu \tau _{\Gamma }B_{\Gamma }\right] .  \tag{3.6}
\end{equation}%
Thus, we obtain%
\begin{equation}
\sinh \varphi =-\lambda \mu \tau _{\Gamma }\frac{ds}{ds^{\star }}\text{ \
and \ }\cosh \varphi =\left( 1+\varepsilon _{1}\varepsilon _{2}\lambda \mu
\kappa _{\Gamma }\right) \frac{ds}{ds^{\star }}  \tag{3.7}
\end{equation}%
and then from (3.4) and (3.7) we obtain,%
\begin{equation}
\tau _{\widetilde{\Gamma }}=\frac{-\widetilde{\varepsilon }_{1}\tau _{\Gamma
}}{1+\varepsilon _{1}\varepsilon _{2}\lambda \mu \kappa _{\Gamma }}. 
\tag{3.8}
\end{equation}

\textbf{Case 1:} If $\Gamma $\ is a timelike curve and it's Mannheim partner
is a timelike curve then $\varepsilon _{1}=\widetilde{\varepsilon }_{1}=-1$
and $\varepsilon _{2}=\widetilde{\varepsilon }_{2}=1$ and from (3.8) we
obtaine%
\begin{equation*}
\tau _{\widetilde{\Gamma }}=\frac{\tau _{\Gamma }}{1-\lambda \mu \kappa
_{\Gamma }}.
\end{equation*}

\textbf{Case 2:} If $\Gamma $\ is a timelike curve and it's Mannheim partner
is a spacelike curve with the timelike principal normal vector then $%
\varepsilon _{1}=-1,$ $\varepsilon _{2}=1$ and$\ \widetilde{\varepsilon }%
_{1}=1,$ $\widetilde{\varepsilon }_{2}=-1$ and from (3.8) we obtaine%
\begin{equation*}
\tau _{\widetilde{\Gamma }}=\frac{-\tau _{\Gamma }}{1-\lambda \mu \kappa
_{\Gamma }}.
\end{equation*}

\textbf{Case 3:} If $\Gamma $\ is a spacelike curve with the timelike
principal normal vector and it's Mannheim partner is a spacelike curve with
the spacelike principal normal vector then the angle $\varphi $ is a real
angle and (3.3) will be%
\begin{eqnarray}
T_{\Gamma } &=&\cos \varphi T_{\widetilde{\Gamma }}+\sin \varphi N_{%
\widetilde{\Gamma }}  \TCItag{3.9} \\
B_{\Gamma } &=&-\sin \varphi T_{\widetilde{\Gamma }}+\cos \varphi N_{%
\widetilde{\Gamma }}  \notag
\end{eqnarray}

From (3.2) and (3.9) we obtain 
\begin{equation}
\frac{ds^{\star }}{ds}=\cos \varphi \text{ \ \ and \ \ }\frac{ds^{\star }}{ds%
}\lambda \tau _{\widetilde{\Gamma }}=\sin \varphi .  \tag{3.10}
\end{equation}

and from (3.6) and compare with the invers of (3.9) we get,%
\begin{equation}
\sin \varphi =\lambda \mu \tau _{\Gamma }\frac{ds}{ds^{\star }}\text{ \ and
\ }\cos \varphi =\left( 1+\varepsilon _{1}\varepsilon _{2}\lambda \mu \kappa
_{\Gamma }\right) \frac{ds}{ds^{\star }}  \tag{3.11}
\end{equation}

By taking $\varepsilon _{1}=1,$ $\varepsilon _{2}=-1$ and$\ \widetilde{%
\varepsilon }_{1}=1,$ $\widetilde{\varepsilon }_{2}=1$ in (3.10) and (3.11)
we obtaine%
\begin{equation}
\tau _{\widetilde{\Gamma }}=\frac{-\tau _{\Gamma }}{1-\lambda \mu \kappa
_{\Gamma }}.  \tag{3.12}
\end{equation}

\textbf{Case 4:} If $\Gamma $\ is a spacelike curve with the spacelike
principal normal vector and it's Mannheim partner is a spacelike curve with
the timelike principal normal vector then $\varepsilon _{1}=1,$ $\varepsilon
_{2}=1$ and$\ \widetilde{\varepsilon }_{1}=1,$ $\widetilde{\varepsilon }%
_{2}=-1$ and from (3.8) we obtaine%
\begin{equation*}
\tau _{\widetilde{\Gamma }}=\frac{-\tau _{\Gamma }}{1+\lambda \mu \kappa
_{\Gamma }}.
\end{equation*}

\textbf{Case 5:} If $\Gamma $\ is a spacelike curve with the spacelike
principal normal vector and it's Mannheim partner is a timelike curve then $%
\varepsilon _{1}=1,$ $\varepsilon _{2}=1$ and$\ \widetilde{\varepsilon }%
_{1}=-1,$ $\widetilde{\varepsilon }_{2}=1$ and from (3.8) we obtaine%
\begin{equation*}
\tau _{\widetilde{\Gamma }}=\frac{\tau _{\Gamma }}{1+\lambda \mu \kappa
_{\Gamma }}.
\end{equation*}

On the other side, we prove (i), (ii), (iii), (iv) for the case 1, 2, 4 and
5. (i) and (ii) are obvious from (3.4) and (3.7) respectively. By taking $%
\frac{ds^{\star }}{ds}=\cosh \varphi $ and $\frac{ds^{\star }}{ds}=\frac{%
\sinh \varphi }{\lambda \widetilde{\varepsilon }_{1}\tau _{\widetilde{\Gamma 
}}}$ from (3.4) in (3.7) we obtain (iii) and (iv). In case 3; (v) and (vi)
are obvious from (3.10) and (3.11) respectively. By taking $\frac{ds^{\star }%
}{ds}=\cos \varphi $ and $\frac{ds^{\star }}{ds}=\frac{\sin \varphi }{%
\lambda \tau _{\widetilde{\Gamma }}}$ from (3.10) in (3.11) we obtain (vii)
and (viii).
\end{proof}

\begin{remark}
Let $\Gamma $ be a NN-FM curve and $\widetilde{\Gamma }$ a NN-FM conjugate
of $\Gamma $ which have nonzero curvature. Thus 
\begin{equation*}
\frac{\widetilde{\varepsilon }_{1}\tau _{\Gamma }+\tau _{\widetilde{\Gamma }}%
}{\kappa _{\Gamma }\tau _{\widetilde{\Gamma }}}
\end{equation*}%
is constant for the cases 1,2,4,5 and 
\begin{equation*}
\frac{\tau _{\Gamma }+\tau _{\widetilde{\Gamma }}}{\kappa _{\Gamma }\tau _{%
\widetilde{\Gamma }}}
\end{equation*}%
is constant for the case 3 too.
\end{remark}

From (3.4), (3.7), (3.10) and (3.11), it is obvious that the angle doesn't
matter real or hyperbolic, between the tangents of NN-FM curves $\Gamma $
and $\widetilde{\Gamma }$ \ is constant such that $\left\langle T_{\Gamma
},T_{\widetilde{\Gamma }}\right\rangle \neq 0$ if and only if $\tau _{\Gamma
}$ and $\kappa _{\Gamma }$ are nonzero constant. This cause that $\widetilde{%
\Gamma }$ is a line and $\Gamma $ is a cicular helix. This can be easly to
prove by using (3.4), (3.7), (3.10) and (3.11). Thus we conclude the
following remark.

\begin{remark}
Let $\Gamma $\ be a NN-FM curve with nonzero curvature and torsiyon and it's
NN-FM conjugate curve be $\widetilde{\Gamma }$. The angle between the
tangents of NN-FM curves $\Gamma $ and $\widetilde{\Gamma }$ \ is constant
such that $\left\langle T_{\Gamma },T_{\widetilde{\Gamma }}\right\rangle
\neq 0$ if and only if $\widetilde{\Gamma }$ is a line and $\Gamma $ is a
cicular helix.
\end{remark}

\begin{theorem}
Let $\Gamma $ be a $C^{\infty }$ Frenet curve with $\tau _{\Gamma }$ nowhere
zero and satisfying the equation for constants $\lambda $ with $\lambda \neq
0$ . Then $\Gamma $ is a non-planar FM curve.
\end{theorem}

\begin{proof}
\textbf{In the cases 1, 2, 4 and 5 : }Define the curve $\Gamma $ with
position vector%
\begin{equation*}
\alpha (s)=\beta (s^{\star })+\lambda (s^{\star })B_{\widetilde{\Gamma }%
}(s^{\star })
\end{equation*}%
Then, denoting differentiation with respect to $s^{\star }$ by a dash, we
have%
\begin{equation*}
\alpha ^{\prime }(s)\frac{ds^{\star }}{ds}=T_{\widetilde{\Gamma }}-%
\widetilde{\varepsilon }_{1}\lambda \tau _{\widetilde{\Gamma }}N_{\widetilde{%
\Gamma }}.
\end{equation*}%
Since $\tau _{\widetilde{\Gamma }}\neq 0$, it follows that $\Gamma $ is a $%
C^{\infty }$ regular curve. Then%
\begin{equation*}
T_{\Gamma }\frac{ds}{ds^{\star }}=T_{\widetilde{\Gamma }}-\widetilde{%
\varepsilon }_{1}\lambda \tau _{\widetilde{\Gamma }}N_{\widetilde{\Gamma }}.
\end{equation*}%
Hence%
\begin{equation*}
\frac{ds}{ds^{\star }}=\sqrt{\widetilde{\varepsilon }_{1}-\widetilde{%
\varepsilon }_{2}\lambda ^{2}\tau _{\widetilde{\Gamma }}^{2}}.
\end{equation*}%
And ,using (3.3)%
\begin{equation*}
T_{\Gamma }=\cosh \varphi T_{\widetilde{\Gamma }}+\sinh \varphi N_{%
\widetilde{\Gamma }}
\end{equation*}%
notice that from (3.3) we have $\sinh \varphi \neq 0$ and constand. Therefore%
\begin{equation*}
\frac{T_{\Gamma }}{ds}\frac{ds}{ds^{\star }}=-\widetilde{\varepsilon }_{1}%
\widetilde{\varepsilon }_{2}\kappa _{\widetilde{\Gamma }}T_{\widetilde{%
\Gamma }}\sinh \varphi +\kappa _{\widetilde{\Gamma }}N_{\widetilde{\Gamma }%
}\cosh \varphi +\tau _{\widetilde{\Gamma }}B_{\widetilde{\Gamma }}\sinh
\varphi
\end{equation*}%
Now define $N_{\Gamma }=\mu B_{\widetilde{\Gamma }}$,%
\begin{equation*}
\kappa _{\Gamma }=\frac{\mu }{\frac{ds}{ds^{\star }}}\tau _{\widetilde{%
\Gamma }}\sinh \varphi .
\end{equation*}%
These are $C^{\infty }$ functions of $s$ (and hence of $s^{\star }$), and%
\begin{equation*}
\frac{T_{\Gamma }}{ds}=\kappa _{\Gamma }N_{\Gamma }.
\end{equation*}

\textbf{In the case 3 : }Define the curve $\Gamma $ with position vector%
\begin{equation*}
\alpha (s)=\beta (s^{\star })+\lambda (s^{\star })B_{\widetilde{\Gamma }%
}(s^{\star })
\end{equation*}%
Then, denoting differentiation with respect to $s^{\star }$ by a dash, we
have%
\begin{equation*}
\alpha ^{\prime }(s)\frac{ds^{\star }}{ds}=T_{\widetilde{\Gamma }}-%
\widetilde{\varepsilon }_{1}\lambda \tau _{\widetilde{\Gamma }}N_{\widetilde{%
\Gamma }}.
\end{equation*}%
Since $\tau _{\widetilde{\Gamma }}\neq 0$, it follows that $\Gamma $ is a $%
C^{\infty }$ regular curve. Then%
\begin{equation*}
T_{\Gamma }\frac{ds}{ds^{\star }}=T_{\widetilde{\Gamma }}-\widetilde{%
\varepsilon }_{1}\lambda \tau _{\widetilde{\Gamma }}N_{\widetilde{\Gamma }}.
\end{equation*}%
Hence%
\begin{equation*}
\frac{ds}{ds^{\star }}=\sqrt{\widetilde{\varepsilon }_{1}-\widetilde{%
\varepsilon }_{2}\lambda ^{2}\tau _{\widetilde{\Gamma }}^{2}}.
\end{equation*}%
And ,using (3.9)%
\begin{equation*}
T_{\Gamma }=\cos \varphi T_{\widetilde{\Gamma }}+\sin \varphi N_{\widetilde{%
\Gamma }},
\end{equation*}%
notice that from (3.9) we have $\sin \varphi \neq 0$ and constand. Therefore%
\begin{equation*}
\frac{T_{\Gamma }}{ds}\frac{ds}{ds^{\star }}=-\kappa _{\widetilde{\Gamma }%
}T_{\widetilde{\Gamma }}\sin \varphi +\kappa _{\widetilde{\Gamma }}N_{%
\widetilde{\Gamma }}\cos \varphi +\tau _{\widetilde{\Gamma }}B_{\widetilde{%
\Gamma }}\sin \varphi
\end{equation*}%
Now define $N_{\Gamma }=\mu B_{\widetilde{\Gamma }}$,%
\begin{equation*}
\kappa _{\Gamma }=\frac{\mu }{\frac{ds}{ds^{\star }}}\tau _{\widetilde{%
\Gamma }}\sin \varphi .
\end{equation*}%
These are $C^{\infty }$ functions of $s$ (and hence of $s^{\star }$), and%
\begin{equation*}
\frac{T_{\Gamma }}{ds}=\kappa _{\Gamma }N_{\Gamma }.
\end{equation*}%
For all cases we define $B_{\Gamma }=T_{\Gamma }\wedge B_{\Gamma }$ and $%
\tau _{\Gamma }=-\left\langle \frac{B_{\Gamma }}{ds},N_{\Gamma
}\right\rangle $. These are also $C^{\infty }$ functions on $I$. It is then
easy to verify that with the frame $\left\{ T_{\Gamma },N_{\Gamma
},B_{\Gamma }\right\} $ and the functions $\kappa _{\Gamma },\tau _{\Gamma }$%
, the curve $\Gamma $ becomes a $C^{\infty }$ Frenet curve. But $N_{\Gamma }$
and $B_{\widetilde{\Gamma }}$ lie on the line joining corresponding points
of $\Gamma $ and $\widetilde{\Gamma }$. Thus $\Gamma $ is a FM curve and $%
\widetilde{\Gamma }$ a FM conjugate of $\Gamma $.
\end{proof}

\begin{theorem}
\bigskip Let $\Gamma $ be a plane $C^{\infty }$ Frenet curve with zero
torsion and whose curvature is either bounded below or bounded above. Then $%
\Gamma $ is a FM curve, and has FB conjugates which are plane curves.
\end{theorem}

\begin{proof}
We prove the theorem in all cases for the curves $\Gamma $ and $\widetilde{%
\Gamma }$. Let $\Gamma $ be a curve satisfying the conditions of the
hypothesis. Then there are non-zero numbers $\lambda $ such that $\kappa _{%
\widetilde{\Gamma }}<-\frac{1}{\lambda }$ on $I^{\star }$ or $\kappa _{%
\widetilde{\Gamma }}>-\frac{1}{\lambda }$ on $I^{\star }$. For any such $%
\lambda $, consider the plane curve $\Gamma $ with position vector$\alpha
(s)=\beta (s^{\star })+\lambda (s^{\star })B_{\widetilde{\Gamma }}(s^{\star
})$%
\begin{equation*}
\beta =\alpha -\mu \lambda N_{\Gamma }.
\end{equation*}%
Then%
\begin{equation*}
T_{\widetilde{\Gamma }}=\left( 1+\varepsilon _{1}\varepsilon _{2}\mu \lambda
\kappa _{\Gamma }\right) T_{\Gamma }
\end{equation*}%
Since $\left( 1+\varepsilon _{1}\varepsilon _{2}\mu \lambda \kappa _{\Gamma
}\right) \neq 0$, $\widetilde{\Gamma }$ is a $C^{\infty }$ regular curve. It
is then a straightforward matter to verify that $\widetilde{\Gamma }$ is a
FM conjugate of $\Gamma $.
\end{proof}

\begin{corollary}
The line joining corresponding points of\ $\Gamma $ and $\widetilde{\Gamma }$
is orthogonal to both $\Gamma $ and $\widetilde{\Gamma }$ at the points $%
s,s^{\star }$ respectively,
\end{corollary}

\section{Non-Null Weakened Mannheim curves}

\begin{definition}
Let $D$ be a subset of a topological space $X$. A function on $X$ into a set 
$Y$ is said to be $D$-piecewise constant if it is constant on each component
of $D$.
\end{definition}

\begin{lemma}
Let $X$ be a proper interval on the real line and $D$ an open subset of $X$.
Then a necessary and sufficient condition for every continuous, $D$%
-piecewise constant real function on $X$ to be constant is that $X\backslash
D$ should have empty dense-in-itself kernel.
\end{lemma}

We notice, however, that if $D$ is dense in $X$, any $C^{1}$ and $D$%
-piecewise constant real function on $X$ must be constant, even if $D$ has
non-empty dense-in-itself kernel.

\begin{theorem}
A NN-WM curve for which $N$ and $Z$ have empty dense-in-itself kernels is a
NN-FM curve.
\end{theorem}

\begin{proof}
Let $\Gamma :\alpha (s),$ $s\in I$ be a NN-WM curve and $\widetilde{\Gamma }$
$:\beta (s^{\star }),$ $s^{\star }\in I^{\star }$ a NN-WM conjugate of $%
\Gamma $. It follows from the definition that $\Gamma $ and $\widetilde{%
\Gamma }$ each has a $C^{\infty }$ family of tangent vectors $T_{\widetilde{%
\Gamma }}(s^{\star })$,$T_{\Gamma }(s)$. Let%
\begin{equation}
\alpha (\sigma (s^{\star }))=\beta (s^{\star })+\lambda (s^{\star })B_{%
\widetilde{\Gamma }}(s^{\star })  \tag{4.1}
\end{equation}%
where $B_{\widetilde{\Gamma }}(s^{\star })$ is some unit vector function and 
$\lambda (s^{\star })\geq 0$ is some scalar function. Let $D^{\star
}=I^{\star }$ $\backslash $ $\sigma (Z)$, $D=I$ $\backslash N$ . Then $%
s(s^{\star })\in $ $C^{\infty }$on $D^{\star }.$

\textbf{Step 1.} To prove $\lambda =$constant for the all cases.

Since $\lambda =\left\Vert \beta (s^{\star })-\alpha (\sigma (s^{\star
}))\right\Vert \geq 0$, it is of class $C^{\infty }$ and it is not only
continuous on $I^{\star }$ but also on every interval of $D^{\star }$ on
which it is nowhere zero. Let $P=\left\{ s^{\star }\in I^{\star }:\lambda
(s^{\star })\neq 0\right\} $ and $X$ any component of $P$ then $P$ and also $%
X$, is open in $I^{\star }$. Let $I^{\star }$ be any component interval of $%
X\cap D^{\star }$ then $\lambda (s^{\star })$ and $B_{\widetilde{\Gamma }%
}(s^{\star })$ are of class $C^{\infty }$ on $I^{\star }$, and from (4.1) we
have%
\begin{equation*}
\alpha ^{\prime }(s)\sigma ^{\prime }(s^{\star })=\beta ^{\prime }(s^{\star
})+\lambda ^{\prime }(s^{\star })B_{\widetilde{\Gamma }}(s^{\star })+\lambda
(s^{\star })B%
{\acute{}}%
_{\widetilde{\Gamma }}(s^{\star }).
\end{equation*}%
Now by definition of a NN-WM curve we have $\left\langle \alpha ^{\prime
}(s),B_{\Gamma }(s)\right\rangle =0=$ $\left\langle \beta ^{\prime
}(s^{\star }),B_{\Gamma }(s)\right\rangle $. Hence, using the identity $%
\left\langle B_{\Gamma }^{\prime }(s),B_{\Gamma }(s)\right\rangle =0$, we
have 
\begin{equation*}
0=\lambda ^{\prime }(s^{\star })\left\langle B_{\widetilde{\Gamma }%
}(s^{\star }),B_{\widetilde{\Gamma }}(s^{\star })\right\rangle .
\end{equation*}%
Therefore $\lambda =$constant on $I^{\star }$ and $\lambda $ is constant on
each interval of the set $X\cap D^{\star }$ too. But by hypothesis $%
X\backslash D^{\star }$ has empty dense-in-itself kernel then it follows
from Lemma 2 that $\lambda $ is constant (and non-zero) on $X$. Since $%
\lambda $ is continuous on $I^{\star }$, $X$ must be closed in $I^{\star }$
but $X$ is also open in $I^{\star }$. Therefore by connectedness we must
have $X=I^{\star }$, that is, $\lambda $ is constant on $I^{\star }$.

\textbf{Step 2.} For the all cases, to prove the existence of two frames 
\begin{equation*}
\left\{ T_{\widetilde{\Gamma }}(s^{\star }),N_{\widetilde{\Gamma }}(s^{\star
}),B_{\widetilde{\Gamma }}(s^{\star })\right\} \text{ and }\left\{ T_{\Gamma
}(\sigma (s^{\star })),N_{\Gamma }(\sigma (s^{\star })),B_{\Gamma }(\sigma
(s^{\star }))\right\}
\end{equation*}%
which are Frenet frames for the NN-WM curves $\Gamma $, $\widetilde{\Gamma }$
on $D$, $D^{\star }$ respectively. Since $\lambda $ is a non-zero constant,
it follows from (4.1) that $B_{\widetilde{\Gamma }}(s^{\star })$ is
continuous on $I^{\star }$ and $C^{\infty }$ on $D^{\star }$, and is always
orthogonal to $T_{\widetilde{\Gamma }}(s^{\star })$. Now define $B_{%
\widetilde{\Gamma }}(s^{\star })=T_{\widetilde{\Gamma }}(s^{\star })\wedge
N_{\widetilde{\Gamma }}(s^{\star })$. Then $\left\{ T_{\widetilde{\Gamma }%
}(s^{\star }),N_{\widetilde{\Gamma }}(s^{\star }),B_{\widetilde{\Gamma }%
}(s^{\star })\right\} $ forms a right-handed orthonormal frame for $%
\widetilde{\Gamma }$ which is continuous on $I^{\star }$ and $C^{\infty }$
on $D^{\star }$.

Now from the definition of NN-WM curve, we see that there exists a scalar
function $\kappa _{\widetilde{\Gamma }}(s^{\star })$ such that $\kappa _{%
\widetilde{\Gamma }}(s^{\star })=\left\langle T%
{\acute{}}%
_{\widetilde{\Gamma }}(s^{\star }),N_{\widetilde{\Gamma }}(s^{\star
})\right\rangle $ on $I^{\star }$ . Hence $\kappa _{\widetilde{\Gamma }%
}(s^{\star })$ is continuous on $I^{\star }$ and $C^{\infty }$ on $D^{\star
} $. Thus the first Frenet formula holds on $D^{\star }$. It is then
straightforward to show that there exists a $C^{\infty }$ function $\tau _{%
\widetilde{\Gamma }}(s^{\star })$ on $D^{\star }$ such that the Frenet
formulas hold. Thus $\left\{ T_{\widetilde{\Gamma }}(s^{\star }),N_{%
\widetilde{\Gamma }}(s^{\star }),B_{\widetilde{\Gamma }}(s^{\star })\right\} 
$ is a Frenet frame for $\widetilde{\Gamma }$ on $D^{\star }$.

Similarly, it can easly to show that there exists a right-handed orthonormal
frame $\left\{ T_{\Gamma }(\sigma (s^{\star })),N_{\Gamma }(\sigma (s^{\star
})),B_{\Gamma }(\sigma (s^{\star }))\right\} $ for $\Gamma $ which is
continuous on $I^{\star }$ and is a Frenet frame for $\widetilde{\Gamma }$
on $D^{\star }$. Moreover, we can choose%
\begin{equation*}
N_{\Gamma }(\sigma (s^{\star }))\text{ =}\pm B_{\widetilde{\Gamma }%
}(s^{\star })\text{\ }
\end{equation*}%
\textbf{Step 3.}

To prove that $N=\varnothing $, $Z=\varnothing $ for the cases 1, 2, 4, 5.
At first we notice that on $D^{\star }$ we have%
\begin{equation*}
\left\langle T_{\widetilde{\Gamma }},T_{\Gamma }\right\rangle ^{\prime
}=\left\langle \kappa _{\widetilde{\Gamma }}N_{\widetilde{\Gamma }%
},T_{\Gamma }\right\rangle +\left\langle T_{\widetilde{\Gamma }},\kappa
_{\Gamma }N_{\Gamma }\sigma ^{\prime }(s^{\star })\right\rangle =0,
\end{equation*}%
so that $\left\langle T_{\widetilde{\Gamma }},T_{\Gamma }\right\rangle $ is
constant on each component of $D^{\star }$ and on $I^{\star }$ \ too, by
Lemma 2. Consequently there exists a constant hyperbolic angle $\varphi $
such that%
\begin{eqnarray*}
T_{\widetilde{\Gamma }}(s^{\star }) &=&T_{\Gamma }(\sigma (s^{\star }))\cosh
\varphi +B_{\Gamma }(\sigma (s^{\star }))\sinh \varphi \\
N_{\widetilde{\Gamma }}(s^{\star }) &=&T_{\Gamma }(\sigma (s^{\star }))\sinh
\varphi +B_{\Gamma }(\sigma (s^{\star }))\cosh \varphi
\end{eqnarray*}%
Further,%
\begin{equation*}
B_{\widetilde{\Gamma }}(s^{\star })=\mu N_{\Gamma }(\sigma (s^{\star })).
\end{equation*}%
Since $\left\{ T_{\widetilde{\Gamma }}(s^{\star }),N_{\widetilde{\Gamma }%
}(s^{\star }),B_{\widetilde{\Gamma }}(s^{\star })\right\} $ are of class $%
C^{\infty }$ on $D^{\star }$ then 
\begin{equation*}
\left\{ T_{\Gamma }(\sigma (s^{\star })),N_{\Gamma }(\sigma (s^{\star
})),B_{\Gamma }(\sigma (s^{\star }))\right\}
\end{equation*}%
are also of class $C^{\infty }$ with respect to $s^{\star }$ on $D^{\star }$%
. Writing (4.1) in the form%
\begin{equation*}
\beta (s^{\star })=\alpha (\sigma (s^{\star }))-\lambda (s^{\star })\mu
N_{\Gamma }(\sigma (s^{\star })).
\end{equation*}%
and differentiating with respect to $s^{\star }$ on $D^{\star }\cap \sigma
^{-1}(D)$, we have%
\begin{equation*}
T_{\widetilde{\Gamma }}(s^{\star })=\sigma ^{\prime }\left[ \left(
1+\varepsilon _{1}\varepsilon _{2}\lambda \mu \kappa _{\Gamma }\right)
T_{\Gamma }+\lambda \mu \tau _{\Gamma }B_{\Gamma }\right] .
\end{equation*}%
Hence%
\begin{equation}
\sigma ^{\prime }\left( 1+\varepsilon _{1}\varepsilon _{2}\lambda \mu \kappa
_{\Gamma }\right) =\cosh \theta \text{ and }\lambda \tau _{\Gamma }=\sinh
\theta .  \tag{4.2}
\end{equation}%
Since $\kappa _{\Gamma }(\sigma (s^{\star }))=\left\langle T_{\Gamma }%
{\acute{}}%
,N_{\Gamma }\right\rangle $ is defined and continuous on $I$ and then $%
\sigma ^{-1}(D)$ is dense, it follows by continuity that (4.2) holds
throughout $D^{\star }$.

Since $\cosh \varphi \neq 0$, then (4.2) implies that $\sigma ^{\prime }\neq
0$ on $D^{\star }$. Hence $Z=\varnothing $. Similarly $N=\varnothing $.

Consequently $\beta (s^{\star })=\alpha (\sigma (s^{\star }))-\lambda
(s^{\star })\mu N_{\Gamma }(\sigma (s^{\star }))$ is of class $C^{\infty }$
on $I$.

To prove that $N=\varnothing $, $Z=\varnothing $ for the case 3. At first we
notice that on $D^{\star }$ we have%
\begin{equation*}
\left\langle T_{\widetilde{\Gamma }},T_{\Gamma }\right\rangle ^{\prime
}=\left\langle \kappa _{\widetilde{\Gamma }}N_{\widetilde{\Gamma }%
},T_{\Gamma }\right\rangle +\left\langle T_{\widetilde{\Gamma }},\kappa
_{\Gamma }N_{\Gamma }\sigma ^{\prime }(s^{\star })\right\rangle =0,
\end{equation*}%
so that $\left\langle T_{\widetilde{\Gamma }},T_{\Gamma }\right\rangle $ is
constant on each component of $D^{\star }$ and on $I^{\star }$ \ too, by
Lemma 2. Consequently there exists a constant real angle $\varphi $ such that%
\begin{eqnarray*}
T_{\widetilde{\Gamma }}(s^{\star }) &=&T_{\Gamma }(\sigma (s^{\star }))\cos
\varphi -B_{\Gamma }(\sigma (s^{\star }))\sin \varphi \\
N_{\widetilde{\Gamma }}(s^{\star }) &=&T_{\Gamma }(\sigma (s^{\star }))\sin
\varphi +B_{\Gamma }(\sigma (s^{\star }))\cos \varphi
\end{eqnarray*}%
Since $\left\{ T_{\widetilde{\Gamma }}(s^{\star }),N_{\widetilde{\Gamma }%
}(s^{\star }),B_{\widetilde{\Gamma }}(s^{\star })\right\} $ are of class $%
C^{\infty }$ on $D^{\star }$ then 
\begin{equation*}
\left\{ T_{\Gamma }(\sigma (s^{\star })),N_{\Gamma }(\sigma (s^{\star
})),B_{\Gamma }(\sigma (s^{\star }))\right\}
\end{equation*}
are also of class $C^{\infty }$ with respect to $s^{\star }$ on $D^{\star }$%
. Writing (4.1) in the form%
\begin{equation*}
\beta (s^{\star })=\alpha (\sigma (s^{\star }))-\lambda (s^{\star })\mu
N_{\Gamma }(\sigma (s^{\star })).
\end{equation*}%
and differentiating with respect to $s^{\star }$ on $D^{\star }\cap \sigma
^{-1}(D)$, we have%
\begin{equation*}
T_{\widetilde{\Gamma }}(s^{\star })=\sigma ^{\prime }\left[ \left( 1-\lambda
\mu \kappa _{\Gamma }\right) T_{\Gamma }+\lambda \mu \tau _{\Gamma
}B_{\Gamma }\right] .
\end{equation*}%
Hence%
\begin{equation}
\sigma ^{\prime }\left( 1-\lambda \mu \kappa _{\Gamma }\right) =\cos \varphi 
\text{ and }-\sigma ^{\prime }\lambda \mu \tau _{\Gamma }=\sin \varphi . 
\tag{4.3}
\end{equation}%
Since $\kappa _{\Gamma }(\sigma (s^{\star }))=\left\langle T_{\Gamma }%
{\acute{}}%
,N_{\Gamma }\right\rangle $ is defined and continuous on $I$ and then $%
\sigma ^{-1}(D)$ is dense, it follows by continuity that (4.3) holds
throughout $D^{\star }$

If\textbf{\ }$\cos \varphi \neq 0$, then (4.3) implies that $\sigma ^{\prime
}\neq 0$ on $D^{\star }$. Hence $Z=\varnothing $. Similarly $N=\varnothing $.

If $\cos \varphi =0$ then%
\begin{equation*}
T_{\widetilde{\Gamma }}(s^{\star })=-B_{\Gamma }(\sigma (s^{\star })).
\end{equation*}%
Differentiation of (4.1) with respect to $s^{\star }$ in $D^{\star }$ gives%
\begin{equation*}
T_{\widetilde{\Gamma }}=\sigma ^{\prime }\lambda \mu \tau _{\Gamma
}B_{\Gamma }.
\end{equation*}%
Hence using (4.3) we have%
\begin{equation*}
\sigma ^{\prime }\lambda \mu \tau _{\Gamma }=-1.
\end{equation*}%
Therefore%
\begin{equation*}
\tau _{\Gamma }=-\frac{1}{\sigma 
{\acute{}}%
\lambda \mu },
\end{equation*}%
and so also on $I^{\star }$, by Lemma 2. It follows that $\tau _{\Gamma }$
is nowhere zero on $I^{\star }$. Consequently $\beta (s^{\star })=\alpha
(\sigma (s^{\star }))-\lambda (s^{\star })\mu N_{\Gamma }(\sigma (s^{\star
}))$ is of class $C^{\infty }$ on $I^{\star }$. Hence $N=\varnothing $.
Similarly $Z=\varnothing $.
\end{proof}

\end{document}